\documentclass[reqno]{amsart}
\usepackage[colorlinks,linkcolor=black,anchorcolor=black,citecolor=black,hyperindex=black,CJKbookmarks=black]{hyperref}

\newtheorem{theorem}{Theorem}[section]
\newtheorem{lemma}[theorem]{Lemma}

\theoremstyle{definition}

\newtheorem{example}[theorem]{Example}

\theoremstyle{remark}
\newtheorem{remark}[theorem]{Remark}
\newtheorem{corollary}[theorem]{Corollary}

\numberwithin{equation}{section}

\begin{document}

\title{A note on R\'{e}nyi's ``record'' problem and Engel's series}

%    Information for first author
\author {Lulu Fang}
\address{Department of Mathematics, South China University of Technology, Guangzhou 510640, P.R. China}
\email{f.lulu@mail.scut.edu.cn}

\author {Min Wu}
\address{Department of Mathematics, South China University of Technology, Guangzhou 510640, P.R. China}
\email{wumin@scut.edu.cn}

%    Information for second author
%\author{Author Two}
%\address{Mathematical Research Section, School of Mathematical Sciences,
%Australian National University, Canberra ACT 2601, Australia}
%\email{two@maths.univ.edu.au}
%\thanks{Support information for the second author.}

%    General info
%\thanks {* Corresponding author}

\begin{abstract}
In 1973, Williams \cite{lesWil73} introduced two interesting discrete Markov processes, namely $C$-processes and $A$-processes, which are related to record times in order statistics and Engel's series in number theory respectively. Moreover, he also showed that these two processes share the same classical limit theorems, such as the law of large numbers, central limit theorem and law of the iterated logarithm. In this paper, we consider the large deviations for these two Markov processes, which indicate that there is a difference between $C$-processes and $A$-processes in the context of large deviations.
\end{abstract}

\maketitle

\section{Introduction and main results}
Let $(\Omega, \mathcal{F}, \mathbf{P})$ be a probability space and we assume that the following processes both define on it.

\subsection{C-processes}

A stochastic process $\{C_n: n \geq 1\}$ taking values in $\mathbb{N}$ will be called a \emph{C-process} if
the sequence $C_0, C_1, C_2, \cdots$ where $C_0 = 1$, forms a time-homogeneous Markov
chain with one-step transition probabilities:
\begin{equation}
p_{ij} = \frac{i}{j(j-1)}\ \text{for any}\ j \geq i+1; \ \ \ \ \  p_{ij} = 0 \ \text{for any}\ j \leq i.
\end{equation}
Next we introduce two interesting examples of $C$-processes in number theory and in order statistics respectively.

\begin{example}[Modified Engel's series]\label{m}
In particular, $\Omega = (0,1)$, $\mathcal{F}$ is the Borel $\sigma$-algebra on $(0,1)$ and $\mathbf{P}$ denotes the
Lebesgue measure on $(0,1)$. In \cite{lesRen62A}, R\'{e}nyi studied the \emph{modified Engel's series} of the form
\begin{equation*}
x = \frac{1}{d_1} + \frac{1}{(d_1 -1)d_2} +  \cdots + \frac{1}{(d_1 -1)\cdots(d_{n-1} -1)d_n} + \cdots,
\end{equation*}
where $d_n \geq 2$ is an integer and $d_{n+1} \geq d_n +1$ for any $n \in \mathbb{N}$. Then the digit sequence
$\{d_n: n \geq 1\}$ is a $C$-process (see R\'{e}nyi \cite[Eq.(5.3)]{lesRen62A}.
\end{example}

\begin{example}[Record times]\label{r}
Let $\{X_n:n \geq 1\}$ be a sequence of independent and identically distributed (i.i.d.)~real-valued random variables with a
common continuous distribution function. The \emph{record times} $L_n$ are defined by $L_0 = 1$ and
\[
L_{n+1} = \min\left\{j: X_j > X_{L_n}, j > L_n\right\} \ \text{for all}\ n \geq 0.
\]
Then $\{L_n: n \geq 1\}$ is a $C$-process (see R\'{e}nyi \cite[Eq.(20)]{lesRen62T}). For more properties and applications of
record times, see \cite{lesCha52, lesGS75, lesPfe87, lesWHK14} and the references therein.
\end{example}

\subsection{A-processes}

We say a stochastic process $\{A_n: n \geq 1\}$ taking values in $\mathbb{N}$ an \emph{A-process} if it
is identical in law to the Engel's series (see Erd\H{o}s et al.~\cite{lesE.R.S58}). That is to say, an $A$-process can be constructed as follows: an irrational number $x$ in $(0,1)$ is chosen at random according to the uniform distribution and expressed uniquely as
\[
x = \frac{1}{A_1} + \frac{1}{A_1A_2} + \cdots+ \frac{1}{A_1 A_2\cdots A_n}+\cdots,
\]
where $A_n$ are integers and $A_n \geq A_{n-1}$ for any $n \in \mathbb{N}$ with the convention that $A_0=2$. Erd\H{o}s et al.~\cite[Eq.(1.3)]{lesE.R.S58} proved that the digit sequence of Engel's series forms a time-homogeneous Markov
chain with one-step transition probabilities:
\begin{equation}\label{A}
p_{ij} = \frac{i-1}{j(j-1)}\ \text{for any}\ j \geq i \geq 2; \ \ \ \ \  p_{ij} = 0 \ \text{for any}\ j < i.
\end{equation}
Since $A$-processes are identical in law to the Engel's series, we know that an $A$-process also forms a time-homogeneous Markov chain with one-step transition probabilities defined as (\ref{A}).

In his paper \cite{lesWil73}, Williams pointed out that the $C$-process and $A$-process fulfill the following representations respectively:
\[
C_0 = 1,\ \ \ \ \ \ \ \  C_n= [C_{n -1} \exp(W_n )] + 1,
\]
\[
a_0 = 1,\ \ \ \ \ \ \ a_n = [a_{n-1} \exp(W_n)],\ \ \ \ \ \ \   A_n = a_n +1,
\]
where $[x]$ denotes the greatest integer not exceeding $x$ and $W_1, W_2, \cdots $ are i.i.d.~exponential random variables with mean 1. Here denote by $B_n^* = W_1 + \cdots + W_n$, Williams also showed that the sequences $\left\{B_n^* - \log C_n\right\}_{n \geq 1}$ and $\left\{B_n^* - \log A_n\right\}_{n \geq 1}$ are almost surely bounded, which implies that $(\log C_n)/n$ and $(\log A_n)/n$ are asymptotically close to $B_n^*/n$, i.e., the empirical mean of i.i.d.~exponential random variables with mean 1.
As a consequence, the strong law of large numbers, the central limit theorem and the law of iterated logarithm for $C$-processes and $A$-processes can be
read off from the classical results for $\{B_n^* , n\geq 1\}$. However, our following results indicate that
the approximation of $(\log C_n)/n$ by $B_n^*/n$ is enough while the approximation of $(\log A_n)/n$ by $B_n^*/n$ is not enough in the  the context of large deviations. This is a difference between these two Markov processes.

\subsection{Main results}
It is worth remarking that the classical limit theorems basically concern that the averages taken over large samples converge to expected values in some sense, but say little or nothing about the rate of convergence. One way to address this is the theory of large deviations in modern probability theory.
Let $\{Y_n: n \geq 1\}$ be a sequence of the real-valued random variables defined on $(\Omega, \mathcal{F}, \mathbf{P})$. A function $I: \mathbb{R} \to [0,\infty]$ is called a \emph{good rate function} if it is lower
semi-continuous and has compact level sets. We say that the sequence $\{Y_n: n \geq 1\}$ satisfies a \emph{large deviation principle} (LDP for short) with speed $n$ and good rate function $I$, if for any Borel set $\Gamma$,
\begin{equation*}
-\inf_{x \in \Gamma^\circ}I(x) \leq \liminf_{n \to \infty} \frac{1}{n}\log \mathbf{P}(Y_n \in \Gamma) \leq \limsup_{n \to \infty} \frac{1}{n}\log \mathbf{P}(Y_n \in \Gamma) \leq -\inf_{x \in \overline{\Gamma}}I(x),
\end{equation*}
where $\Gamma^\circ$ and $\overline{\Gamma}$ denotes the interior and the closure of $\Gamma$ respectively. For instance, we can study the probability that the empirical mean of a sequence of random variables deviates away from its ergodic mean.  In general, these probabilities are exponentially small and follow a large deviation principle.
For an introduction to the theory of large deviations, we refer the reader to Dembo and Zeitouni \cite{lesD.Z1998} and Touchette \cite{lesTou09}.

\begin{theorem}\label{LDPC}
Let $\{C_n: n \geq 1\}$ be a $C$-process. Then $\left\{\frac{\log C_n - n}{n}: n \geq 1\right\}$ satisfies an LDP with speed $n$ and good rate function
\begin{equation}\label{rate function C}
I_C(x)=
\begin{cases}
x - \log (x +1),  &\text{if $x >-1$};\\
+\infty,  & \text{if $x \leq-1$}.
\end{cases}
\end{equation}
\end{theorem}

As an application of Theorem \ref{LDPC}, we immediately obtain the large deviations for modified Engel's series and record times by Examples \ref{m} and \ref{r}.

\begin{corollary}\label{RE}
The sequence $\left\{\frac{\log d_n - n}{n}: n \geq 1\right\}$ or $\left\{\frac{\log L_n - n}{n}: n \geq 1\right\}$ satisfies an LDP with speed $n$ and good rate function defined as (\ref{rate function C}). \end{corollary}

\begin{remark}
It is worth pointing out that Gut \cite{lesGut90, lesGut02} has considered the following tail probabilities for record times
\[
\mathbf{P}\big(|\log L_n -n| \geq n\varepsilon\big),
\]
where $\varepsilon>0$ and $n \in \mathbb{N}$. Moreover, he showed that these tail probabilities are exponentially
small (see \cite[Theorem 6]{lesGut90}). However, the Corollary \ref{RE} indicates that the probability of the rate event that $(\log L_n)/n$ deviates away from 1 decays to zero exponentially with an explicit rate function.
\end{remark}

Since an $A$-process is identical in law to the Engel's series, the large deviation principle for $A$-processes can be stated as follow (see Zhu \cite[Theorem 1.2]{lesZhu2014}).

\begin{theorem}(\cite{lesZhu2014})\label{LDPA}
Let $\{A_n: n \geq 1\}$ be an $A$-process. Then $\left\{\frac{\log A_n - n}{n}: n \geq 1\right\}$ satisfies an LDP with speed $n$ and good rate function
\begin{equation}\label{A RATE}
I_A(x)=
\begin{cases}
x - \log (x +1),  &\text{if $x >-\frac{1}{2}$};\\
-x - 1 + \log 2,  &\text{if $-1 \leq x \leq -\frac{1}{2}$};\\
+\infty,  & \text{if $x < -1$}.
\end{cases}
\end{equation}
\end{theorem}

\begin{remark}
As we have seen in Theorem \ref{LDPC}, the rate function $I_C$ completely coincides with the rate function for
$B_n^*/n$ (see \cite[Example 3.2]{lesTou09}). However, the rate function $I_A$ in Theorem \ref{LDPA} coincides
with the rate function for $B_n^*/n$ only when $x >-1/2$. That is to say, the approximation of $(\log A_n)/n$
by the empirical mean of i.i.d.~exponential random variables with mean 1 is not enough in the context of large deviations (see Zhu \cite[Remark 1.4]{lesZhu2014}).
\end{remark}

\section{Proof of Main Results}
We use the notation $\mathbf{E}(\xi)$ to denote the expectation of a random variable $\xi$ with respect to the probability measure $\mathbf{P}$.
In this section, we will give the proofs of Theorems \ref{LDPC} and \ref{LDPA}. In fact, the large deviations for stochastic processes related Markov chain
occurring in number theory have been discussed by Zhu \cite{lesZhu2014}, Fang \cite{lesFangSPL, lesFangJNT}, and Fang and Shang \cite{lesFS}. More precisely, Zhu \cite{lesZhu2014} studied the large deviations for Engel's series whose digit sequence is non-decreasing and forms a time-homogeneous Markov chain with exact one-step transition probabilities.
Fang and Shang \cite{lesFS} considered the large deviations for Engel continued fractions whose digit sequence is also non-decreasing but does not form a Markov chain. Since the digit sequences of Engel's series and Engel continued fractions are non-decreasing, their rate functions have the similar structure in form.
In \cite{lesFangSPL} and \cite{lesFangJNT}, the author mainly investigated large deviations for two expansions of real numbers whose digit sequence both are strictly increasing and forms a time-homogeneous Markov chain with exact transition probabilities.
So we will use the methods in \cite{lesFangSPL} and \cite{lesFangJNT} to prove our main result. To do this we need an elementary lemma.

\begin{lemma}\label{inequality1}
Let $\theta <1$ be a real number. For any $j \geq 1$, we have
\begin{equation*}
\left(1+\frac{1}{j}\right)^{\theta -1}\cdot \frac{1}{1-\theta} \leq \sum_{k=j+1}^{\infty}\frac{j}{k(k-1)}\left(\frac{k}{j}\right)^\theta \leq \left(1+\frac{1}{j}\right)\cdot \frac{1}{1-\theta}.
\end{equation*}
\end{lemma}

\begin{proof}
The proof is similar with the proof of Lemma 3.1 in \cite{lesFangJNT}.
\end{proof}

Recall that $C_0, C_1, C_2, \cdots$ is a strictly increasing stochastic process and also forms a time-homogeneous Markov chain. Lemma \ref{inequality1} and the methods of Lemma 3.2 in \cite{lesFangJNT} immediately yield that

\begin{lemma}\label{inequality2}
\begin{equation*}
\lim_{n \to \infty}\frac{1}{n}\log \mathbf{E}(C_n^\theta) =
\begin{cases}
\log \frac{1}{1-\theta}, &\text{if $\theta < 1$}; \\
+\infty , & \text{if $\theta \geq 1$}.
\end{cases}
\end{equation*}
\end{lemma}

Now we are ready to prove Theorem \ref{LDPC}.

\begin{proof}[Proof of Theorem \ref{LDPC}]
For any $\theta \in \mathbb{R}$, we define the logarithmic moment generating function (see Dembo and Zeitouni \cite{lesD.Z1998}) for $C$-processes as
\[
 \Lambda_C(\theta) = \lim_{n \to \infty} \frac{1}{n}\log \mathbf{E}\left(\exp{\left(\frac{\log C_n -n}{n}\cdot n\theta\right)}\right)
\]
if the limit exists. In view of Lemma \ref{inequality2}, we know that
\[
\Lambda_C(\theta) = - \theta + \lim_{n \to \infty}\frac{1}{n}\log \mathbf{E}(C_n^\theta) =
\begin{cases}
-\theta - \log(1-\theta), &\text{if $\theta < 1$}; \\
+\infty,  & \text{if $\theta \geq 1$}.
\end{cases}
\]
By the G\"{a}rtner-Ellis theorem (see \cite[Theorem 2.3.6]{lesD.Z1998}), we obtain that the sequence $\left\{\frac{\log C_n -n}{n}: n\geq 1\right\}$ satisfies an LDP with speed $n$ and good rate function, i.e., the Legendre transform of $\Lambda_C(\theta)$,
\begin{equation*}
I_C(x) = \sup_{\theta \in \mathbb{R}}\left\{\theta x - \Lambda_C(\theta)\right\}\ \text{for all}\ x \in \mathbb{R}. \end{equation*}
Now we will show that $I_C(x)$ is as defined in the Theorem \ref{LDPC}. Since $\Lambda_C(\theta) =+\infty$ if
$\theta \geq 1$, then $I_C(x)$ is alternatively given by
\[
I_C(x) = \sup_{\theta < 1}\left\{\theta x + \theta + \log(1-\theta)\right\}\ \text{for all}\ x \in \mathbb{R}.
\]
For any $x \in \mathbb{R}$, let $f(\theta) = \theta x + \theta + \log(1-\theta)$ for all $\theta < 1$. It is worth
pointing out that $f$ is strictly concave and that $\theta = x/(x+1)$ is the unique maximal point of it.
When $x \leq -1$, we know that $f(\theta) = \theta (x +1)+ \log(1-\theta)$ is decreasing with respect to $\theta$ on the interval $(-\infty, 1)$ and hence that $I_C(x)= +\infty$. If $x >-1$, we have that $x/(x+1)<1$. That is to say,
the unique maximal point of $f$ locates in its domain. So, $I_C(x)= x -\log (x+1)$ for any $x >-1$. Thus, we determined the rate function $I_C(x)$ defined as
\begin{equation*}
I_C(x)=
\begin{cases}
x - \log (x +1),  &\text{if $x >-1$};\\
+\infty,  & \text{if $x \leq-1$}.
\end{cases}
\end{equation*}
\end{proof}

\begin{proof}[Proof of Theorem \ref{LDPA}]
Since the proof is similar to the proof of Theorem \ref{LDPC}, we confine ourselves to providing a sketch.
Although $A$-processes are non-decreasing (not strictly increasing) Markov chains, the methods in the proof of Theorem \ref{LDPC} are still valid for $A$-processes if we properly modify the corresponding proofs of Lemmas \ref{inequality1} and \ref{inequality2}.
In fact, we can obtain
\begin{equation*}
\lim_{n \to \infty}\frac{1}{n}\log \mathbf{E}(A_n^\theta) =
\begin{cases}
\max\left\{-\log 2,\log \frac{1}{1-\theta}\right\}, &\text{if $\theta < 1$}; \\
+\infty , & \text{if $\theta \geq 1$},
\end{cases}
\end{equation*}
(see \cite[Remark 6]{lesFS} for more details).
Replacing $C$ by $A$ in the relevant objects defined as in the proof of Theorem \ref{LDPC}, we actually deduce that
\[
\Lambda_A(\theta) =
\begin{cases}
-\theta - \log 2, &\text{if $\theta \leq -1$}; \\
-\theta - \log(1-\theta), &\text{if $ -1<\theta < 1$}; \\
+\infty,  & \text{if $\theta \geq 1$}.
\end{cases}
\]
By the G\"{a}rtner-Ellis theorem, we know the sequence $\left\{\frac{\log A_n -n}{n}: n\geq 1\right\}$ satisfies an LDP and its rate function $I_A$ given by (\ref{A RATE}) which is the Legendre transform of $\Lambda_A(\theta)$ in fact. It is worth remarking that $\log 2$ appears in this formula for $A$-processes by comparing with the result of Lemma \ref{inequality2}. This fact eventually leads to the difference of rate functions for $C$-processes and $A$-processes.
\end{proof}

\section{Futher Remarks}
This section is devoted to giving some further remarks about the differences between $C$-processes and $A$-processes.

1. It is clear to see $I_{C}(x) \geq I_{A}(x)$ for any $x \in \mathbb{R}$. Strict inequality holds for the interval $[-1,-1/2)$. This shows that the decay speed of the probability for which $(\log C_n)/n$ deviates away from 1 is faster than that of $(\log A_n)/n$. In fact, to understand the difference between $I_{C}$ and $I_{A}$, we take value of $x$ between $-1$ and $-1/2$ and then $A_n =2$ implies $(\log A_n -n)/n \leq x$ for sufficiently large $n$. By the definition of $A$-processes, we have $\mathrm{P}(A_n =2) = 2^{-n}$ and hence that
\[
\limsup_{n \to \infty}\frac{1}{n} \log \mathrm{P}\left(\frac{\log A_n -n}{n} \leq x\right) \geq  -\log 2,
\]
which implies $I_A(-1) \leq \log 2$, whereas this type of behaviour is ruled out for the $C$-process.

2. Now we consider the large deviation problems for $C$-processes with $C_0 = c~(c \in \mathbb{N})$ and $A$-processes with $A_0 = a~(a \geq 2$ and are integers), our methods in this paper imply that this class of $C$-processes also satisfy an LDP and their rate function are still $I_C(x)$ defined as (\ref{rate function C}); while these $A$-processes satisfy an LDP depending on the initial value $A_0=a$, more precisely,
the good rate function is
\begin{equation*}
I_{A,a}(x)=
\begin{cases}
x - \log (x +1),  &\text{if $x >-1+1/a$};\\
(1-a)(x + 1) + \log a,  &\text{if $-1 \leq x \leq -1+1/a$};\\
+\infty,  & \text{if $x < -1$},
\end{cases}
\end{equation*}
see Zhu \cite[Remark 1.3]{lesZhu2014}. Clearly, $I_A=I_{A,2}$ and $I_{A,a} \to I_C$ as $a \to +\infty$. This indicates that the initial value does not impact the large deviation result for $C$-processes, but it will change the rate function of large deviations for $A$-processes.

\section{Conclusions}
Williams \cite{lesWil73} used a strong approximation method for two interesting discrete Markov processes, namely $C$-processes and $A$-processes, and proved that these two processes share the classical limit theorems, even other laws which hinge on the \emph{Invariance Principle}. As a consequence, the Engel's series and modified Engel's series of real numbers have the same classical limit theorems. Moreover, these two different representations of real numbers also have the same fractal structure (see Liu and Wu \cite{lesLW03}, and Wang and Wu \cite{lesWW06}). However, our main results indicate that they have a difference in the context of large deviations. From this point of view, large deviation is a more useful tool in such questions.

In fact, the fractal and large deviation are two main aspects of dealing with the almost everywhere results.
The former is concerned with the fractal structure (such as, Hausdorff dimension) of the set of points for which the desired almost everywhere result does not hold or attains any other values; 
while the latter considers the speeds of probabilities of the events that the desired almost everywhere result deviates away from its ergodic mean. 
What is interesting is that these two different ways have some relations in some special cases (see Denker and Kesseb\"{o}hmer \cite{lesDK01}, Fang et al.~\cite{lesFWLa,lesFWL}, Pesin and Weiss \cite{lesPW01}).

{\bf Acknowledgement}
We wish to thank the referee for the helpful comments.
The authors also would like to thank Professor Jonathan Warren and Professor Sandy Davie for useful discussions and suggestions.

\end{document}